\documentclass[11pt]{article}

\usepackage{listings}
\usepackage{epsfig}
\usepackage{lscape}
\usepackage{longtable}
\usepackage{amsmath,amssymb,amsthm}
\usepackage{graphicx}
\usepackage{color}
\usepackage{placeins}
\usepackage{url}
\usepackage{cases}
\usepackage{hyperref}
\usepackage{algorithm}  

\DeclareMathOperator*{\argmin}{argmin}

\def\approxleq{ \kern3pt \mbox{\raisebox{.6ex}{$<$}} \kern-8pt
  \mbox{\raisebox{-.6ex}{$\sim$}} \kern5pt}
\def\mc{\multicolumn}
\def\norm#1{\|#1 \|}
\def\inprod#1#2{\langle#1,\,#2 \rangle}

   \def\cK{{\cal K}}
   \def\cT{{\cal T}}
\def\cH{{\cal H}}     \def\cO{{\cal O}}
 \def\cM{{\cal M}} \def\cN{{\cal N}} 
\def\cS{{\cal S}}     \def\cV{{\cal V}}
\def\cI{{\cal I}}   
 \def\cN{{\cal N}}

\newlength{\len}
\newtheorem{theorem}{Theorem}
\newtheorem{proposition}{Proposition}
\newtheorem{lemma}{Lemma}
\newtheorem{corollary}{Corollary}
\newtheorem{remark}{Remark}

\newtheorem{definition}{Definition}

\begin{document}

\title{\bf Fast projection onto the ordered weighted $\ell_1$ norm ball}

\author{Qinzhen Li\thanks{School of Data Science, Fudan University, Shanghai, China ({\tt 18210980031@fudan.edu.cn}).}, \;
	 Xudong Li\thanks{School of Data Science, Fudan University, Shanghai, China ({\tt lixudong@fudan.edu.cn}); Shanghai Center for Mathematical Sciences, Fudan University, Shanghai, China.}
}

\date{\today}
\maketitle

\begin{abstract}
	In this paper, we provide a finitely terminated yet efficient approach to compute the Euclidean projection onto the ordered weighted $\ell_1$ (OWL1) norm ball. In particular, an efficient semismooth Newton method is proposed for solving the dual of a reformulation of the original projection problem. Global and local quadratic convergence results, as well as the finite termination property, of the algorithm are proved. Numerical comparisons with the two best-known methods demonstrate the efficiency of our method. In addition, we derive the generalized Jacobian of the studied projector which, we believe, is crucial for the future designing of fast second order nonsmooth methods for solving general OWL1 norm constrained problems.
\end{abstract}
\noindent
\textbf{Keywords:}
Euclidean projector, ordered weighted $\ell_1$ norm ball,  HS-Jacobian, semismooth Newton method

\medskip
\noindent
\textbf{AMS subject classifications:}
90C06, 90C20, 90C25
\section{Introduction}\label{intro}
Given a nonzero vector $\lambda \in \Re^n$ satisfying $\lambda_1 \ge \lambda_2 \ge \ldots \ge \lambda_n \ge 0$, the ordered weighted $\ell_1$ (OWL1) norm $\kappa_\lambda$ is defined as
\[ \kappa_{\lambda}(x) = \inprod{|x|^{\downarrow}}{\lambda} = \sum_{i=1}^n \lambda_i |x|_i^{\downarrow}, \quad \forall \, x\in\Re^n, \]
where $|x|\in\Re^n$ denotes the vector whose $i$th entry is $|x_{i}|$ and for any given vector $z\in\Re^n$, $z^{\downarrow}$ denotes
the vector of entries of $z$ being arranged in the non-increasing order $z_1^\downarrow \ge \cdots \ge z_n^\downarrow$. {It can be shown that $\kappa_{\lambda}$ is a norm function if only if $\lambda_1 >0$ (see \cite[Exercise IV.1.19]{Bhatia1997Matrix}).} In fact, both the $\ell_1$ norm and the $\ell_\infty$ norm
can be obtained from $\kappa_{\lambda}$ via choosing different weighted vectors $\lambda$.
Given $\tau >0$, define the $\tau$ level-set of $\kappa_{\lambda}$ as $\Delta_{\tau}:= \left\{x\in\Re^n \mid \kappa_{\lambda}(x) \le \tau  \right\}$.
In this paper, we focus on the problem of projecting a vector $b$  onto $\Delta_{\lambda}$, i.e., solving the following optimization problem: 
\begin{equation}  \label{prob:proj}
\Pi_{\Delta_{\tau}}(b):= \argmin_x \left\{\frac{1}{2} \norm{x - b}^2 \mid \kappa_{\lambda}(x) \le \tau  \right\}. 
\end{equation}
Obviously, if $\kappa_{\lambda}(b)\le \tau$, we have trivially that $\Pi_{\Delta_{\tau}}(b) = b$. To avoid this special case, throughout this paper, we make the blank assumption that   
 $\kappa_{\lambda}(b) > \tau$.

Recently, the OWL1 norm function $\kappa_{\lambda}$ receives significant academic attention in both optimization and statistical communities. Coined as the weighted Ky Fan norm function, $\kappa_{\lambda}$ appeared in the optimization literature \cite{Wu2014on} where the first and second order properties of the Moreau–Yosida regularization of $\kappa_{\lambda}$ were carefully studied. In \cite{Wu2014on}, the authors also emphasized the important roles played by the class of Ky Fan norm functions in matrix optimization problems.  Later, as a generalization of the octagonal
shrinkage and clustering algorithm for regression (OSCAR) regularizer, the OWL1 norm was studied as a group sparsity promoted regularizer in \cite{Zeng2014Decreasing}. Meanwhile, an OWL1 norm regularized least squares model, termed as the sorted L-one penalized estimation (SLOPE), was proposed in \cite{Bogdan2015Slope}. The authors showed that the SLOPE estimator enjoys nice statistical properties. In particular, it has a controllable false discovery rate, a measure of type-I error for multiple testing. Recently, Su et al. in \cite{Su2016Slope} further argued that under certain mild assumptions, the SLOPE estimator achieves asymptotic minimaxity over
large sparsity classes for Gaussian designs.  Hence, as a fundamental subroutine for solving the OWL1 norm constrained estimator which is a provable statistically well-behaved estimator, the computation of the projector $\Pi_{\Delta_{\tau}}$ in \eqref{prob:proj} deserves more research efforts.

As an important regularizer, the computation of the proximal mapping of $\kappa_{\lambda}$, defined as follows
\[
{\rm Prox}_{\kappa_{\lambda}}(b) := \argmin_x \left\{
\kappa_{\lambda}(x) + \frac{1}{2}\norm{x - b}^2
\right\}, \quad\forall\, b\in \Re^n,
\]
was studied in  \cite{Zeng2014Decreasing,Bogdan2015Slope,Zeng2014ordered}. These studies revealed that the computation of ${\rm Prox}_{\kappa_{\lambda}}$ can be reduced to the Euclidean projection problem over the monotone nonnegative cone\footnote{See Subsection \ref{subsec:reform} for the definition.} and can be efficiently solved by the famous pool adjacent violators algorithm (PAVA) \cite{Kruskal1964Nonmetric,Barlow1972statistical,Best1990active}. With this effective subroutine, the accelerated proximal gradient method \cite{Nesterov1983method} was used to solve the SLOPE model in \cite{Bogdan2015Slope}. Recently, the computational efficiency for the SLOPE model, especially for large-scale difficult problems, was greatly enhanced by Luo et al. \cite{Luo2019Solving} in which the authors carefully exploited the special low-rank and sparsity structures in the generalized Jacobian of ${\rm Prox}_{\kappa_{\lambda}}$ and designed a highly efficient sparse semismooth Newton based augmented Lagrangian method. 

Unlike the proximal mapping ${\rm Prox}_{\kappa_{\lambda}}$, the Euclidean projector $\Pi_{\Delta_{\tau}}$ is less understood. In fact, given the discussions in the literature of classic level-set methods \cite{Aravkin2019level}, one could argue that the computations of ${\rm Prox}_{\kappa_{\lambda}}$ and $\Pi_{\Delta_{\tau}}$ are mathematically equivalent.
However, as far as we know, there is no closed-form expression
in the literature to explicitly characterize this equivalence.
Nevertheless, this equivalence can be used to design algorithms to compute $\Pi_{\Delta_{\tau}}$ from ${\rm Prox}_{\kappa_{\lambda}}$. For example, in \cite{Zeng2014ordered}, the authors advocated a root-finding iterative scheme based on the equivalence.
Later, Davis in \cite{damek2015algorithm} noticed that the root-finding scheme can not obtain $\Pi_{\Delta_{\tau}}$ in a finite number of steps and proposed a finitely terminated approach for computing the projector. Partitioning non-decreasing entries of the preprocessed input into groups, Davis' approach conducts a series of tests, and partial sortings and averagings. However, as shown in Section \ref{sec:numerical}, it can be time-consuming in solving high dimensional problems. We also note that neither \cite{Zeng2014ordered} and \cite{damek2015algorithm} studied the generalized Jacobian of $\Pi_{\Delta_{\tau}}$, which can be essential for the algorithmic design for solving general OWL1 norm constrained optimization problems as demonstrated in \cite{Li2018efficiently,Li2020efficient,Luo2019Solving,Lin2019efficient,Zhang2020efficient}.

In this paper, we aim to design a finitely terminated yet efficient algorithm for large-scale projection problem \eqref{prob:proj}.  Our proposed algorithm starts with an observation that problem \eqref{prob:proj} can be reduced to the problem of finding the projector onto the intersection of the monotone nonnegative cone and an affine subspace. Building upon the HS-Jacobian \cite{Han1997Newton, Li2020efficient} of the projector onto the monotone nonnegative cone, we are able to design an efficient semismooth Newton ({\sc Ssn}) method to solve the dual of the reduced problem.
The global and local quadratic convergence results, as well as the finite termination properties, of the {\sc Ssn} are provided. Numerical experiments demonstrate that our algorithm is faster than the root-finding approach and outperforms Davis' approach by a large margin. In addition to these, we conduct a careful first order variational analysis on $\Pi_{\Delta_{\tau}}$ and construct its generalized Jacobian which lays the foundation for the subsequent designing of second order nonsmooth methods for solving general OWL1 norm constrained problems.

The remaining parts of this paper are organized as follows. In the next section, we review some definitions and results which will be used later for designing and analyzing our algorithm. In Section \ref{sec:dualssn}, we recast the original problem as the problem of projecting onto the intersection of the monotone nonnegative cone and the affine subspace and derive the dual of the reformulated problem. 
Then, a semismooth Newton method is proposed for this dual problem. Theoretical convergence results, as well as the finite termination property, of our algorithm are also provided in this section. 
Section \ref{sec:GJacobianPiDeltatau} is dedicated to the generalized Jacobian of $\Pi_{\Delta_{\tau}}$.  In Section \ref{sec:numerical},
we conduct numerical experiments to evaluate the performance of our algorithms against the root-finding and Davis' approaches. We conclude our paper in the last section.

Before moving to the next section, we list here some notation to be used later in this paper. For any given vector $y \in \Re^n$, we use ${\rm Diag}(y)$ to denote the diagonal matrix whose $i$th diagonal element is $y_i$, and $e_n$ is used to denote vectors of all ones in $\Re^n$. We denote by $O_n$ and $I_n$ the $n\times n$ zero matrix and identity matrix, respectively. For any given matrix $A \in \Re^{m\times n}$, we denote by $\rm Null(A)$ the null space of $A$.

\section{Preliminaries}
We briefly recall in this section some basic definitions and results associated with the set-valued mappings and the semismoothness of vector-valued functions.

Let $\cK:\Re^n \rightrightarrows\Re^n$ be a set-valued mapping. The domain of $\cK$, denoted ${\rm dom}\,\cK$, is the set ${\rm dom}\, \cK \equiv \left\{
x\in\Re^n \mid \cK(x) \neq \emptyset
\right\}$. The set-valued mapping $\cK$ is said to be nonempty valued at a given point $p$ if $p\in {\rm dom}\, \cK$. We also recall the following definitions associated with $\cK$. 

\begin{definition}[{\cite[Definition 2.1.16]{facchinei2007finite}}]
	\label{def:setvalued}
A set-valued mapping $\cK:\Re^n \rightrightarrows \Re^n$ is said to be
\begin{enumerate}
	\item closed at point $\bar x$ if there are sequences $\{x^k\}$ and $\{y^k\}\subseteq \Re^n$ such that
	$x^k\to \bar x$ as $k\to \infty$ and for all $k\ge 0$, $y^k\in \cK(x^k)$ and $y^k \to \bar y$ as $k\to \infty$, then $\bar y\in \cK(\bar x)$;
	\item bounded at point $\bar x$ if $\cK(\bar x)$ is bounded;
	\item upper semicontinuous at point $\bar x$ if for every open set $\cV$ containing $\cK(\bar x)$, there exists an open neighborhood $\cN$ of $\bar x$ such that, for each $x\in \cN$, $\cV$ contains $\cK(x)$;
	\item upper semicontinuous on a set $\Omega \subseteq \Re^n$ if 
	$\cK$ is  upper semicontinuous at every point in $\Omega$.
\end{enumerate}
\end{definition}

\begin{definition}[B-subdifferential and Clarke generalized Jacobian \cite{Clarke1983optimization}]
Let $F:\Re^n \to \Re^n$ be a locally Lipschitz continuous
function. Define the B-subdifferential of $F$ at point $x\in \Re^n$ to be 
\[
\partial_B F(x) = \left\{ V\in \Re^{n\times n} \mid V = \lim_{x^k \to x} JF(x^k), \, x^k\in \Omega_F \right\},
\]
where $\Omega_F: = \left\{ x\in\Re^n \mid \mbox{$F$ is differentiable at $x$} \right\}$ and define the Clarke generalized Jacobian of $F$ at point $x$ as 
\[\partial F(x) = {\rm conv}\partial_B F(x),\]
i.e., the convex hull of $\partial_B F(x)$.
\end{definition}
By \cite[Propositions 7.1.4 and 7.4.11]{facchinei2007finite}, we see that both the Clarke generalized Jacobian $\partial F$ and the B-subdifferential $\partial_B F$ are nonempty and compact valued, upper-semicontinuous set-valued mappings.

Inspired by the definitions given in \cite{mifflin1977semismooth,kummer1988newton,qi1993nonsmooth,sun2002semismooth}, we can state the following definition of semismooth functions with respect to any given nonempty and compact valued, upper-semicontinuous set-valued mapping. In fact, this definition has been used in recent papers for solving various high-dimensional statistical optimization problems \cite{Li2018efficiently,Lin2019efficient,Zhang2020efficient}.

\begin{definition}[Semismoothness]
	\label{def:semismooth} {Let $\cO \subseteq \Re^n$ be an open set,
		$\cK:\cO \subseteq \Re^n \rightrightarrows\Re^{m\times n}$  be a nonempty and  compact valued,} upper-semicontinuous set-valued mapping and $F:\cO \rightarrow \Re^m$ be a locally Lipschitz continuous function. $F$ is said to be semismooth at $x\in \cO$ with respect to the multifunction $\cK$ if $F$ is directionally differentiable at $x$ and
	for any {$V \in  \cK(x + \Delta x)$ with $\Delta x\rightarrow 0$,}
	\[F(x+\Delta x) - F(x) - V\Delta x = o(\norm{\Delta x}).\]
	{Let $\gamma$ be a positive constant. $F$ is said to be $\gamma$-order (strongly, if $\gamma =1$) semismooth at $x$ with respect to $\cK$ if $F$ is   directionally differentiable at $x$ and for any  $V \in  \cK(x + \Delta x)$ with $\Delta x\rightarrow 0$,
		\[F(x+\Delta x) - F(x) - V\Delta x = O(\norm{\Delta x}^{1+\gamma}).\]
		$F$ is said to be a  semismooth (respectively, $\gamma$-order semismooth, strongly semismooth) function on $\cO$ with respect to  $\cK$
		if it is semismooth (respectively, $\gamma$-order semismooth,  strongly semismooth) everywhere in $\cO$ with respect to $\cK$.}
\end{definition}
Semismooth functions and the corresponding algorithms have been extensively studied and used in the nonsmooth optimization/equation community. Indeed, it is well known that continuous piecewise affine functions and twice continuously differentiable functions are all strongly semismooth everywhere. We refer readers to \cite{facchinei2007finite} for more examples of semismooth functions.
The following second order limit result is in fact an extension of \cite[Theorem 2.1]{pang1995globally} and has been proved and used in \cite{Li2018efficiently}. We will also use it to prove the convergence of our algorithm.
\begin{proposition}\label{prop:taylor2}
	Let $\theta:\Omega\to \Re$ be a continuously differentiable function and its gradient $\nabla \theta:\Omega \to \Re^n$ is locally Lipschitz {on the open set $\Omega$.} If $\nabla\theta$ is semismooth at a point $x\in\Omega$ with respect to {a nonempty, compact valued and upper-semicontinuous multifunction $\cK: \Omega \rightrightarrows\cS^n$,} then for any $V\in\cK(x+d)$ with
	$d\to 0$, {we have}
	\[\theta(x+d) - \theta(x) - \inprod{\nabla \theta(x)}{d} - \frac{1}{2}\inprod{d}{Vd} = o(\norm{d}^2).\]
\end{proposition}

\section{A dual semismooth Newton algorithm for solving (\ref{prob:proj})}
\label{sec:dualssn}
In this section, by some well-known reduction techniques, we are able to represent the OWL1 norm constraint in a more explicit form and to transform the original problem \eqref{prob:proj} into a handy reformulation. Then, we propose a semismooth Newton method to efficiently solve the reformulated projection problem.

\subsection{A reformulation of problem (\ref{prob:proj})}
\label{subsec:reform}
We focus first on the reformulation of the original problem \eqref{prob:proj}. 
Let ${\bf \Pi}_{\bf n}^{\bf s}$ be the set of  all signed permutation matrices in $\Re^{n\times n}$. For any given vector $u\in \Re^n$, denote by
\begin{equation}
\label{eq:Pmatrix}
\Pi^s(u):=\left\{ P\in {\bf \Pi}_{\bf n}^{\bf s} \mid Pu = |u|^{\downarrow} \right\},
\end{equation}
the set of signed permutation matrices associated with $u$.
Let $B\in \Re^{n\times n}$ be the matrix defined by $Bx = [x_1 - x_2, x_2 - x_3, \ldots, x_{n-1} - x_n, x_n]^T$, for all $x\in\Re^n$ and define the monotone nonnegative cone as
\[C:=\{x\in\Re^n\mid x_1\ge\cdots\ge x_n\ge 0 \} = \{ x\in \Re^n \mid Bx \ge 0\}.\] 
Given $\lambda \in C$ and $\tau >0$, define $C_{\lambda}^\tau := \left\{ x\in \Re^n \mid \inprod{\lambda}{x} = \tau, \, x\in C \right\}$ and 
\begin{equation}
\label{prob:newp}
\Pi_{C_{\lambda}^\tau}(b):= \argmin_x\left\{
\frac{1}{2}\norm{x - w}^2 \mid x\in C_{\lambda}^\tau
\right\}, \quad b\in\Re^n.
\end{equation}
The following proposition states that the original projection problem \eqref{prob:proj} can be reformulated in a more explicit form via involving the monotone nonnegative cone $C$.
\begin{proposition}
\label{prop:reform}
For any given $b\in \Re^n$, suppose that $\kappa_{\lambda}(b) > \tau$. Then, it holds that 
\[
\Pi_{\Delta_{\tau}}(b) = P^T \Pi_{C_{\lambda}^\tau}(Pb), \quad \forall\, P\in \Pi^s(b).
\] 
\end{proposition}
Note that Proposition \ref{prop:reform} can be proved via simple reduction techniques. In fact, although presented in different ways, it has been proved in \cite{Zeng2014ordered,damek2015algorithm}. For the compactness of the current paper, we provide a simple proof for this proposition in the Appendix.

Proposition \ref{prop:reform} implies that we can, without loss of generality, assume in problem \eqref{prob:newp} that $b\in C$.
As one can observe later, our algorithm and analysis in fact works for problem \eqref{prob:newp} with a general vector $b\in\Re^n$.
In order to design efficient algorithms for solving problem \eqref{prob:newp}, we first study its constraints system, i.e., 
\begin{equation}
\label{eq:cons}
\inprod{\lambda}{x} = \tau, \quad x\in C.
\end{equation}
A simple observation shows that the Slater's condition holds, i.e.,
there exists $\hat x \in {\rm int}(C)$ such that
$\inprod{\lambda}{\hat x} = \tau$. For example, one can set $\hat x_i = \tau/\norm{\lambda}_1$ for $i=1,\ldots,n$.
The following nondegeneracy condition associated with \eqref{eq:cons} stems from the seminar works of Robinson \cite{Robinson1984local,Robinson1987local,Robinson2003constraint} and has been extensively studied for semidefinite programming problems \cite{Sun2006strong,Chan2008constraint}. It can be regarded as a generalization of the classic linear independent constraint qualification (LICQ) \cite{Robinson1984local,Shapiro2003sensitivity}.
\begin{definition}
	\label{def:nondegen}
	We say that a feasible solution $\tilde x$ to system \eqref{eq:cons} is constraint nondegenerate if
	\begin{equation}
	\label{eq:nondegen}
	\lambda^T {\rm lin}\cT_C(\tilde x) = \Re,
	\end{equation}
	where $\cT_C(\tilde x)$ denotes the tangent cone of $C$ at point $\tilde x$ and ${\rm lin}\cT_C(\tilde x)$ represents the lineality space of $\cT_C(\tilde x)$.
\end{definition}

\begin{proposition}
	\label{prop:nondegen}
	It holds that any feasible solution to problem \eqref{prob:newp} is constraint nondegenerate in the sense of Definition \ref{def:nondegen}.
\end{proposition}
\begin{proof}
	For any feasible solution $\tilde x$ to \eqref{prob:newp}, it holds that
	\[
	\cT_C(\tilde x) = \{d\in \Re^n \mid (Bd)_i \ge 0, \, i \in \cI(\tilde x) \},
	\]
	where $\cI(\tilde x)$ is the set of active indices given as
	$\cI(\tilde x) = \{i\in \{1,\dots,n\}\mid (B\tilde x)_i = 0\}$. Thus, simple calculations show that
	\[
	{\rm lin} \cT_C(\tilde x) = \{d\in \Re^n \mid (Bd)_i = 0, i\in \cI(\tilde x)\} = {\rm Null}(B_{\cI(\tilde x)}),
	\]
	where $B_{\cI(\tilde x)}$ is the submatrix obtained by extracting rows of $B$ with indices in $\cI(\tilde x)$. We only consider the nontrivial case where $ \cI(\tilde x)\neq \emptyset$. 
	Since $\tilde x \neq 0$, we know that $ \cI(\tilde x) \neq [n]$. 
	
	Next, we show that 
	matrix $\begin{bmatrix}
	\lambda^T \\
	B_{\cI(\tilde x)}
	\end{bmatrix}$ is of full row rank.	
	Suppose, on the contrary, that there exists a nonzero vector $\mu \in\Re^{|\cI(\tilde x)|}$ such that $\lambda =\sum_{i\in I} B_{i,:}^T \mu_i$.
	Denote $i_0 := \inf\{i\in [n] \mid i\not\in \cI(\tilde x)  \}$, i.e., the smallest index which is not in $\cI(\tilde x)$.
	Let $\tilde e \in \Re^n$, be the vector whose first $i_0$ entries are all ones with the rest entries being zeros.
	Then, we have 
	\[ 0 < \sum_{i=1}^{n-1} \lambda_i = \inprod{\tilde e}{\lambda}
	= \sum_{i\in \cI(\tilde x)} \mu_i \inprod{\tilde e}{B_{i,:}^T} = 0,   \]
	where the last equality follows from the fact that 
	$\inprod{\tilde e}{B_{i,:}^T} = 0$ for all $i\in [n]$.
	Hence, we arrive at a contradiction.
	
	Now, the above full row rankness implies that, for any $u\in\Re$, there exists $d\in\Re^n$ such that
	\[\lambda^T d = u, \quad B_{\cI(\tilde x)}d = 0,\]
	i.e., $\lambda^T {\rm lin}\cT_C(\tilde x) = \Re$. The proof is thus completed.	
\end{proof}
\begin{remark}
	\label{remk:licq}
	As can be observed from the proof of Proposition \ref{prop:nondegen}, the nondegeneracy condition \eqref{eq:nondegen} is in fact equivalent to the classic linear independent constraint qualification associated with the following system
	\[
	\inprod{\lambda}{x} = \tau, \quad Bx \ge 0.
	\]
\end{remark}

\subsection{A dual semismooth Newton method for (\ref{prob:newp})}
We note that problem \eqref{prob:newp} concerns about computing the Euclidean projector onto the intersection of an affine subspace and a closed convex cone. Various algorithms, including the dual quasi-Newton method \cite{Malic2004dual}, the dual semismooth Newton method \cite{Qi2006qudratically} and the alternating
projections method with Dykstra’s correction \cite{Dykstra1983algorithm,Higham2002computing}, have been proposed for computing the projector. Extensive numerical comparisons in \cite{Qi2006qudratically} show that the dual semismooth Newton approach outperforms others significantly. 
Hence, in this subsection, we focus on designing a semismooth Newton method for solving the dual of \eqref{prob:newp} and study its convergence properties. 

The dual of problem \eqref{prob:newp} in the minimization form can be written as follows:
\begin{equation}
\label{eq:d}
\min_y \left\{\phi(y):= \frac{1}{2}\norm{\Pi_C(y \lambda + b)}^2 - y\tau -\frac{1}{2}\norm{b}^2 \right\}.
\end{equation}
Since the Slater's condition associated with problem \eqref{prob:newp} is satisfied, we know that $\phi$ is coercive, i.e., $\phi(y)\to +\infty$ as $|y|\to +\infty$ \cite{Rockafellar1974conjugate}.
Hence, we know that the solution set $\Omega_{D} \subseteq \Re^n$ of the dual problem \eqref{eq:d} is nonempty. Meanwhile, {since $C$ is a closed convex cone, we know that $\norm{\Pi_{C}(\cdot)}^2$ is a continuously differentiable function (see for example \cite[Theorem 31.5]{rockafellar1970convex}).}
Then, it is not difficult to see that $\phi$ is continuously differentiable over $\Re$ with
\begin{equation}  
\label{eq:dphi}
\phi'(y) = \inprod{\Pi_{C}(y\lambda+b)}{\lambda} - \tau. 
\end{equation}
Therefore, the optimal solution set $\Omega_{D}$ of problem \eqref{eq:d} is exactly the solution set of
the following univariate nonlinear nonsmooth equation
\begin{equation} \label{prob:first}
\phi'(y) = 0, \quad y\in\Re.
\end{equation}
Let $y^*\in \Omega_D$ be an solution to \eqref{prob:first}, i.e, an optimal solution to problem \eqref{eq:d}. Then, we know from the optimality condition that $\Pi_{C}(y^*\lambda + b)$ is the unique optimal solution to problem \eqref{prob:newp}. Since $C$ is a polyhedral convex set, by \cite[Proposition 4.1.4]{facchinei2007finite}, we know that $\phi'$ is in fact a Lipschitz continuous piecewise affine function on $\Re$.
Hence, $\phi'$ is semismooth on $\Re$ with respect to the classic Clarke generalized Jacobian $\partial \phi'$  and a semismooth Newton method can be used for efficiently solving the above equation. 

\subsubsection{Computations of $\Pi_C$ and its generalized Jacobian}
\label{subsec:hspic}
In order to use the semismooth Newton method to solve \eqref{prob:first}, we first need to compute $\Pi_C$ efficiently. For any $d\in\Re^n$, recall from the definition of $C$ that
\begin{equation}
\label{prob:projC}
\Pi_{C}(d): = \argmin_x \left\{\frac{1}{2}\norm{x - d}^2 \mid Bx \ge 0 \right\}.
\end{equation}
Due to the special structure of $B$, problem \eqref{prob:projC} has been extensively studied in the optimization and statistical literature. Indeed, problem \eqref{prob:projC} is closely related to the so-called isotonic regression problem and can be solved by the famous pool adjacent violators algorithm (PAVA) \cite{Kruskal1964Nonmetric,Barlow1972statistical,Best1990active}. In this paper, we use a highly efficient implementation of the PAVA provided in \cite{Bogdan2015Slope}.

Next, we shall study the first order variational properties of $\Pi_C$, i.e., obtaining certain generalized Jacobian associated with $\Pi_{C}$ which will be used in the semismooth Newton method.
As the Euclidean projector over a polyhedral convex set, the first order variation properties of $\Pi_C$ have been intensively studied in \cite{Haraux1977how,Pang1990Newton,Pang1996Piecewise,Han1997Newton,Li2020efficient}. It is noted in the above references that the computation of generalized Jacobian associated with a Euclidean projector over a general polyhedral can be numerically detrimental. In \cite{Han1997Newton}, a more tractable replacement of the generalized Jacobian was proposed and was termed as HS-Jacobian in \cite{Li2020efficient} where the authors also derived an explicit formula and efficient approaches for constructing a special HS-Jacobian. 
Since then, the concept of HS-Jacobian has been widely used in the algorithmic design for solving many high-dimensional machine learning problems, for example, the fused lasso problems \cite{Li2018efficiently}, the SLOPE models \cite{Luo2019Solving} and the clustered lasso problems \cite{Lin2019efficient}. The following definition of the HS-Jacobian associated with $\Pi_C$
is based on the study conducted in \cite{Han1997Newton,Li2020efficient}.
Consider the set-valued mapping $\cH:\Re^n \rightrightarrows \Re^{n\times n}$ defined by
\begin{equation} 
\label{eq:gJacobianCH}
\cH(d) := \left\{
H \in \Re^{n\times n} \mid H = I_n - B_{\Gamma}^T(B_{\Gamma}B_{\Gamma}^T)^{-1} B_{\Gamma}, \, \Gamma\in \cK(d)
\right\},
\end{equation}
where $\cK(d):=\{\Gamma\subseteq\{1,\ldots,n\}\mid {\rm Supp}(z(d)) \subseteq \Gamma\subseteq  \cI(\Pi_{C}(d)) \}$. Here, $z(d):= (B^T)^{-1}(d - \Pi_{C}(d))$ is the optimal solution to the dual of the projection problem; ${\rm Supp}(z(d))$ denotes the support of $z(d)$, which is the set of indices $i$ such that $z(d)_i \neq 0$, and 
$$\cI(\Pi_{C}(d)) := \left\{
i\in\{1,\ldots n\}\mid (B\Pi_{C}(d))_i = 0
\right\}$$ 
is the set of active indices at $\Pi_{C}(d)$; $B_\Gamma$ is the submatrix obtained by extracting the rows of $B$ with indices in $\Gamma$ with the convention that ${B_{\Gamma}^T(B_{\Gamma}B_{\Gamma}^T)^{-1} B_{\Gamma}} = 0\in\Re^{n\times n}$ if $\Gamma = \emptyset$. For any given nonempty index set $\Gamma \subseteq \{1,\ldots, n\}$, it is not difficult to observe that $B_\Gamma$ is of full row rank and $B_\Gamma B_\Gamma^T$ is invertible. Hence, the set-valued mapping $\cH$ is well-defined over $\Re^n$. Here, we term $\cH(d)$ as the HS-Jacobian associated with $\Pi_C$ at point $d$. Meanwhile, the mapping $\cH$ has the following desirable property.  

\begin{lemma}\label{lem:partialproj}
	For any $d \in \Re^n$, there exists a neighborhood $W$ of $d$ such that for all $d' \in W$
	\begin{equation}
	\label{eq:le_cH1}
	\cK(d') \subseteq \cK(d), \quad \cH(d') \subseteq \cH(d).
	\end{equation}
	When $\cH(d') \subseteq \cH(d)$, it holds that
	\begin{equation}\label{eq:le_cH2}
	\Pi_{C}(d')-\Pi_{C}(d)-H(d'-d)=0, \quad \forall\, H \in \cH(d'). 
	\end{equation}
	In fact, $\cH(d) \equiv \partial_B\Pi_{C}(d)$, i.e., $\cH(d)$ is exactly the B-subdifferential of $\Pi_C$ at point $d$.
\end{lemma}
\begin{proof}
	The first part of the lemma (\eqref{eq:le_cH1} and \eqref{eq:le_cH2}) follows directly from \cite[Lemma 2.1]{Han1997Newton}. Since the LICQ associated with problem \eqref{prob:projC} holds at $\Pi_{C}(d)$, the last statement holds from \cite[Section 2]{Sun1997computable} and \cite[Theorem 3.2 and Corollary 3.2.2]{Outrata1995numerical}.
\end{proof}

Given $d\in\Re^n$, although the formula of the HS-Jacobian $H\in \cH(d)$ defined in \eqref{eq:gJacobianCH} seems complicated, we show in the following discussion that $H$ in fact enjoys a very simple yet elegant representation which is largely inherited  from the special structure of $B$. 
The discussions below are based on \cite[Proposition 6]{Li2018efficiently} and \cite[Section 3.4]{Luo2019Solving}. For any given nonempty index set $\Gamma \subseteq \{1,\ldots,n\}$, define the diagonal matrix $\Sigma_{\Gamma}\in \Re^{n \times n}$ by
\[(\Sigma_{\Gamma})_{ii}\;=\;\begin{cases}
1, &{\rm if}\; i \in \Gamma,\\
0,&{\rm otherwise}.
\end{cases} \]
We note that $\Sigma_{\Gamma}$ can be partitioned into $N$ blocks so that the diagonal elements of these submatrices are all identity matrices or zero matrices and any two consecutive blocks are different in their types: 
\begin{equation}
\label{eq:Sigam}
\Sigma_{\Gamma} = {\rm Diag}(\Lambda_1,\cdots,\Lambda_N),
\end{equation}
where $\Lambda_j \in \{O_{n_i},I_{n_i}\}$, and $\sum_{i=1}^N n_i = n$. Let $H = I_n-B_{\Gamma}^T(B_{\Gamma}B_{\Gamma}^T)^{-1}B_{\Gamma}$, i.e., the orthogonal projection onto the null space of $B_{\Gamma}$. 
Then, by some direct calculations, we see that $H$ is a nonnegative block diagonal matrix, i.e., 
\begin{align*}
H = {\rm Diag}(H_1,\cdots,H_N),
\end{align*}
where \begin{equation}
\label{eq:Hi}
 H_i = \begin{cases}
\frac{1}{n_i+1} e_{n_i+1}e_{n_i+1}^T, &{\rm if}\; i \in J\;{\rm and}\;i\;\neq\;N,\\
O_{n_i},&{\rm if}\; i \in J \;{\rm and}\; i\;= N,\\
I_{n_i-1},&{\rm if}\; i \notin J \;{\rm and}\; i\;\neq\; 1,\\
I_{n_i},&{\rm if}\; i \notin J \;{\rm and}\; i\;=\;1
\end{cases}
\end{equation}
with $J:=\{j \in \{1,\cdots,N\} \mid \Lambda_j = I_{n_j}\}$ and the convention $I_0 = \emptyset$. In fact, we can further decompose $H$ into the sum of a diagonal matrix and a low rank matrix, i.e., $H = D + UU^T$, where $D = {\rm Diag}(D_1,\dots,D_N)$ with
\[D_i\;=\;\begin{cases}
O_{n_i+1}, &\quad {\rm if}\; i \in J \; {\rm and} \; i\;\neq\;N,\\
O_{n_i}, &\quad {\rm if}\; i \in J \;{\rm and} \; i\;=\;N,\\
I_{n_i-1},&\quad {\rm if}\; i \notin J \;{\rm and} \; i\;\neq\; 1,\\
I_{n_i},&\quad {\rm if}\; i \notin J \;{\rm and} \; i\;=\;1,
\end{cases}\]
and $U \in \Re^{n\times N}$ with its $(k,j)^{th}$ entry given by
\[U_{k,j}=\begin{cases}
\frac{1}{\sqrt{n_j+1}}, &\quad {\rm if}\; \sum_{t=1}^{j-1}n_t+1\leq k\leq\sum_{t=1}^{j}n_t+1 \;{\rm and} \; j\in J\backslash\{N\},\\
0,&\quad {\rm otherwise}.
\end{cases}
\]
\begin{lemma}
	\label{lem:h11p}
For any $d\in\Re^n$ satisfying $\Pi_C(d)\neq 0$, it holds that
\[
H_{11} > 0, \quad \forall\, H\in \cH(d).
\]
\end{lemma}
\begin{proof}
Since $\Pi_C(d)\neq 0$, we know that $\cI(\Pi_C(d)) \neq \left\{1,\ldots,n\right\}$. 
Hence, it holds that $\Gamma \neq \left\{1,\ldots, n \right\}$ for all $\Gamma \in \cK(d)$. If $\Gamma = \emptyset$, we have that $H = I$ and $H_{11} = 1$. If $\Gamma \neq \emptyset$,
we know that $N > 1$ with $N$ being the number of blocks in \eqref{eq:Sigam}. The desired result then follows directly from  \eqref{eq:Hi}. 
\end{proof}
\subsubsection{A semismooth Newton method for (\ref{prob:first}) and its convergence}
From Lemma \ref{lem:partialproj}, we obtain the B-subdifferential of $\Pi_C$. However, it is still a nontrivial task for computing the B-subdifferential or the Clarke generalized Jacobian of $\phi'$ defined in \eqref{eq:dphi}. Here, we propose to circumvent this difficulty by designing a set-valued mapping which can be regarded as an alternative surrogate of $\partial \phi'$. Given a nonzero vector $\lambda \in C$ and $b\in\Re^n$, define the set-valued mapping $\cM:\Re \rightrightarrows \Re$ by
\begin{equation}
\label{eq:cM}
\cM(y) : = \left\{ M\in \Re\mid 
M = \lambda^T H \lambda, \, H\in \cH(y\lambda + b) 
\right\},
\end{equation}
where $\cH$ is the set-valued mapping defined in \eqref{eq:gJacobianCH}. We show in the next proposition that the Lipschitz continuous piecewise affine function $\phi'$ is $\gamma$-order semismooth on $\Re$ with respect to $\cM$ for any given $\gamma >0$ in the sense of Definition \ref{def:semismooth} and $\cM$ is a legitimate replacement of the Clarke generalized Jacobian of $\phi'$.
\begin{proposition}
	\label{prop:cM}
Let $b\in\Re^n$, $y\in\Re$ and a nonzero vector $\lambda\in C$ be any given data. Then, $\cM$, defined in \eqref{eq:cM}, is a nonempty and  compact valued and upper-semicontinuous set-valued mapping and
$M\ge 0$ for all $M\in \cM(y)$. Moreover, there exists a neighborhood $W$ of $y$ such that for all $y'\in W$,
\begin{equation} \label{eq:semismoothphip}
{	\cM(y')\subseteq \cM(y)\quad {\rm and} \quad \phi'(y')- \phi'(y) - M (y' - y) = 0, \quad \forall\, M \in \cM(y').}
\end{equation}
Therefore, $\partial_B\phi'(y) \subseteq \cM(y)$.
\end{proposition}

\begin{proof}
	By Lemma \ref{lem:partialproj}, we know that $\cM$ is a nonempty and  compact valued and upper-semicontinuous set-valued mapping. 
	From the definitions of $\cH$ and $\cM$ in \eqref{eq:gJacobianCH} and \eqref{eq:cM}, it is easy to see that $M\ge 0$ for all $M\in \cM(y)$. Meanwhile, for any given $y\in\Re$, from Lemma \ref{lem:partialproj}, we know that there exists a neighborhood
	$W$ of $y$ such that for all $y'\in W$, $\cH(y'\lambda + b)\subseteq \cH(y\lambda + b)$ and 
\begin{equation}
\label{eq:pic}
\Pi_C{(y'\lambda + b)} - \Pi_C(y\lambda + b) - P\lambda(y'-y) =0, \quad \forall\, P\in \cH(y'\lambda + b).
\end{equation}
Hence, we know from \eqref{eq:cM} that $\cM(y') \subseteq \cM(y)$ for all $y'\in W$. Let $y'$ be a fixed but arbitrary vector in $W$. Then, for any $M\in \cM(y')$, there exists $P\in \cH(y'\lambda + b)$ such that $M = \lambda^T P \lambda$ and 
\begin{align*}
\phi'(y') - \phi'(y) = {}& \inprod{\Pi_C(y'\lambda + b) - \Pi_C(y\lambda + b)}{\lambda} 
= \lambda^T P \lambda (y'-y) = M(y'-y), 
\end{align*}
where the second equation follows from \eqref{eq:pic}. Since $y'$ is chosen arbitrarily from $W$, we know that \eqref{eq:semismoothphip} holds. Then, \eqref{eq:semismoothphip}, together with  the definition of the B-subdifferential of $\phi'$ at $y$, further implies that  $\partial_B\phi'(y) \subseteq \cM(y)$.
\end{proof}
Next, given $y\in\Re$, we show how to construct an element in $\cM(y)$. Let the index set ${\Gamma} = \cI(\Pi_C(y\lambda + b))$, i.e., the set of active indices at $\Pi_{C}(y\lambda + b)$. Then,  
$
{H} = I_n - B_{ \Gamma}^T(B_{ \Gamma} B_{ \Gamma}^T)^{-1}B_{ \Gamma} \in \cH(y\lambda + b)
$ and it holds from the definition of $\cM$ in \eqref{eq:cM} that
\begin{equation}
\label{eq:elementM}
 M = \lambda^T  H \lambda \in \cM( y).
\end{equation}
In the following proposition, we show that all the elements in $\cM(y)$ are positive for all $y$ such that $\Pi_{C}(y\lambda + b) \neq 0$. Proposition \ref{prop:positiveM} will be critical for the algorithmic design and for establishing the convergence properties of the algorithm.
\begin{proposition}
	\label{prop:positiveM}
Given a nonzero vector $\lambda \in C$ and $b\in\Re^n$, let $S = \{y\in\Re\mid \Pi_C(y\lambda + b) \neq 0 \}$. Then, $S$ is nonempty and for any $y\in S$, it holds that
\[
M >0, \quad \forall\, M\in \cM(y),
\]
where the set-valued mapping $\cM$ is defined in \eqref{eq:cM}.
\end{proposition}
\begin{proof}
We first show that $S\neq \emptyset$. Since $\lambda \in C$ is nonzero, it holds that $\lambda_{11} > 0$. Hence, there exists $\hat y >0$ large enough such that $\hat y \lambda_1 + b_1 >0$. 
Consider a vector $\alpha \in \Re^n$ with $\alpha_1 = \hat y \lambda_1 + b_1$ and $\alpha_i =0 $ for all $i=2,\ldots,n$. Clearly, $\alpha \in C$. Observe that
\[
\norm{\alpha - (\hat y\lambda+b)}^2 = \sum_{i=2}^{n} (\hat y\lambda_i + b_i)^2
< \sum_{i=1}^{n} (\hat y\lambda_i + b_i)^2 = \norm{\hat y\lambda + b}^2.
\]
Hence, $\hat y\in S$, i.e., $S$ is nonempty.

For any $y\in S$, since $\Pi_C(y\lambda + b) \neq 0$, we know that 
 $\cI(\Pi_C(y\lambda + b)) \neq \left\{ 1,\ldots, n\right\}$. By Lemma \ref{lem:h11p}, we have that $H_{11} >0$ for all $H\in \cH(y\lambda+b)$. Therefore, simple calculations show that
\[
\lambda^T H \lambda \ge H_{11} \lambda_1^2 > 0, \quad \forall\, H\in \cH(y\lambda+b). 
\]
The desired conclusion then follows from the definition of $\cM$ in \eqref{eq:cM}.
\end{proof}

\begin{corollary}
	\label{coro:Mystar}
For any given nonzero vector $\lambda\in C$, $b\in\Re^n$ and $\tau >0$, problem \eqref{eq:d} has a unique optimal solution, denoted $y^*$, i.e., the optimal solution set $\Omega_D$ is a singleton. Moreover, $M >0$ for all $M\in \cM(y^*)$.
\end{corollary}
\begin{proof}
We note that $\Omega_D$ is exactly the solution set of the univariate piecewise affine equation \eqref{prob:first}. Since $\tau >0$, 
$\Pi_{C}(y\lambda + b) \neq 0$ for any $y\in \Omega_D$. Meanwhile, from Proposition \ref{prop:cM}, we have that $\partial \phi'(y) = {\rm conv}\partial_B \phi'(y) \subseteq {\rm conv}(\cM(y))$. By Proposition \ref{prop:positiveM}, we know that for any $y\in \Omega_D$, all the elements in $\partial \phi'(y)$ are positive.
Thus, the Clarke inverse function theorem \cite[Proposition 7.1.19]{facchinei2007finite} implies that $\Omega_D $ is a singleton.
\end{proof}

With all the above preparations, we are ready to present a semismooth Newton method for solving \eqref{prob:first}.

\medskip

\centerline{\fbox{\parbox{\textwidth}{
			{\bf Algorithm {\sc Ssn}}: {\bf A semismooth Newton algorithm for solving \eqref{prob:first}}
			\\[5pt]
			Set $\mu \in (0,1/2)$ and $\delta\in(0,1)$ and choose $y^0 \in \Re$. Iterate the following steps for $j=0,1,\ldots.$
			\begin{description}
				\item[Step 1.]  Compute $\Pi_{C}(y^j\lambda+b)$ and set \begin{equation*}
				d^j = \left\{
				\begin{aligned}
				& -\phi'(y^j), \quad \mbox{if } \Pi_{C}(y^j\lambda+b) = 0,\\
				& -\phi'(y^j)/M_j, \quad \mbox{otherwise},
				\end{aligned}
				\right.
				\end{equation*} 
				where $M_j \in \cM(y^j\lambda + b)$ is constructed as in \eqref{eq:elementM}. 
				\item[Step 2.] (Line search)  Set $\alpha_j = \delta^{m_j}$, where $m_j$ is the first nonnegative integer $m$ for which
				\begin{equation*}
				\phi(y^j + \delta^{m}d^j) \leq \phi(y^j) + \mu \delta^{m} 
				\phi'(y^j) d^j.
				\end{equation*}
				\item[Step 3.] Set
				$y^{j+1} = y^j + \alpha_j \, d^j.$
			\end{description}
}}}

\medskip

\begin{theorem}
	\label{them:convergenceSSN}
Let  $\{y^j\}$ be the infinite sequence generated by Algorithm {\sc Ssn}.
Then $\{y^j\}$ converges to the unique optimal solution $y^*$ of  problem \eqref{eq:d}. Moreover, it holds that for all $j$ sufficiently large, $y^{j+1} = y^j + d^j$ with $d^j = -\phi'(y^j)/M_j$ and $
|y^{j+1} -  y^*| = O(| y^{j}- y^*| ^2).$
\end{theorem}

\begin{proof} 
By Proposition \ref{prop:positiveM}, we know that $d^j$ is always well-defined and a descent direction. Thus, Algorithm {\sc Ssn} is well-defined. Since $\phi$ is a coercive function, the sequence $\{y^j\}$ is bounded. Then, the standard convergence analysis \cite[Proposition 1.2.1]{Bertsekas1991nonlinear}
implies that any cluster point of $\{y^j\}$ is  a stationary point of problem \eqref{eq:d}, and hence an optimal solution since $\phi$ is convex. Corollary \ref{coro:Mystar} further implies that $y^j$ converges to the unique optimal solution $y^*$ to \eqref{eq:d}.

Since every element in $\cM(y^*)$ is positive and $\cM$ is upper-semicontinuous, from \cite[Lemma 7.5.2]{facchinei2007finite}, Proposition \ref{prop:positiveM} and Corollary \ref{coro:Mystar}, we have that for all $j$ sufficiently large, $\{M_j^{-1}\}$ is uniformly bounded and $d^j = -\phi'(y^j)/M_j$.
{Since $\phi'$ is strongly semismooth with respect to $\cM$,} similar to the proof for \cite[Theorem 3.5]{zhao2010newton}, it can be shown that for all $j$ sufficiently large,
\begin{equation}\label{eq:yorder2}
|y^j + d^j - y^*|
= \cO(|y^j - y^*|^2),
\end{equation}
and for some constant $\hat \delta >0$,
$ - \phi'(y^j)d^j \ge \hat \delta |d^j|^2.$
Based on \eqref{eq:yorder2}, Proposition \ref{prop:taylor2} and \cite[Proposition 8.3.18]{facchinei2007finite},  we know that for $\mu \in (0,1/2)$, there exists an integer $j_0$ such that for all $j \ge j_0$,
\[\phi(y^j + d^j) \le \phi(y^j) + \mu \phi'(y^j){d^j},\]
i.e.,
$y^{j+1} = y^j + d^j$ for all $j\ge j_0$.
This, together with \eqref{eq:yorder2}, completes the proof.	
\end{proof}

The above theorem asserts the local fast convergence of Algorithm {\sc Ssn}. A careful analysis further shows that Algorithm {\sc Ssn} proposes a finite termination property, i.e., if an iterate $y^j$ is sufficiently close to the optimal solution $y^*$, Algorithm  {\sc Ssn} finds $y^*$ in one step.  
\begin{theorem}
	\label{thm:finiteSSN}
Let  $\{y^j\}$ be the infinite sequence generated by Algorithm {\sc Ssn}. Then, there exists a positive integer $j_0$ such that 
\[ y^j = y^*, \quad \forall\, j>j_0.\]
\end{theorem}
\begin{proof}
	By Theorem \ref{them:convergenceSSN} and Proposition \ref{prop:cM}, we know that there exists a positive integer $ j_0$ such that for all $j\ge j_0$, all elements in $\cM(y^j)$ are positive, $d_j = -\phi'(y^j)/M_j$ and
	\[y^{j+1} = y^j + d^j \quad {\rm and } \quad 
	\phi'(y^j)- \phi'(y^*) - M_j (y^j - y^*) = 0, \quad \forall\, M_j \in \cM(y^j).
	\]
	Then, it holds that for all $j\ge j_0$
	\begin{align*}
		y^{j + 1} - y^* ={}& y^{j} + d^{j} - y^*\\
		 ={}&-(\phi'(y^j)- M_j (y^j - y^*))/M_{j} \\
		 = {}&-(\phi'(y^j)- \phi'(y^*) - M_j (y^j - y^*))/M_{j} = 0.
	\end{align*}
	We thus complete the proof.
\end{proof}

\section{Generalized Jacobian of $\Pi_{\Delta_\tau}$}
\label{sec:GJacobianPiDeltatau}
Besides the fast computation of $\Pi_{\Delta_{\tau}}$, another important task is to obtain certain generalized Jacobian associated with this projector. As is demonstrated in \cite{Li2018efficiently,Li2020efficient,Luo2019Solving,Lin2019efficient,Zhang2020efficient}, this concept will be critical to the design of efficient nonsmooth second order algorithms for solving general OWL1 constrained optimization problems. Similar to our analysis of the projector $\Pi_{C}$, a first thought would be trying to derive the HS-Jacobian of $\Pi_{\Delta_\tau}$. However, as far as we know, there is no 
explicit polyhedral representation for the OWL1 norm constraint
$\{x\in \Re^n \mid \kappa_{\lambda}(x)\le \tau\}$ in the literature, which we believe is a nontrivial task. In this section, we propose to overcome this difficulty by exploiting the composition structure of $\Pi_{\Delta_\tau}$ in Proposition \ref{prop:reform} and investigating the HS-Jacobian of the solution mapping $\Pi_{C_{\lambda}^\tau}$ associated with the reformulated problem \eqref{prob:newp} in which the polyhedral constraint enjoys a simple representation.

We first study the HS-Jacobian of $\Pi_{C_{\lambda}^\tau}$. Similar to Subsection \ref{subsec:hspic}, the HS-Jacobian of $\Pi_{C_{\lambda}^\tau}$ defined in \eqref{prob:newp} at any given point $w\in\Re^n$  can be written as follows:
\[
\cV(w):= \left\{
V\in \Re^{n\times n} \mid V = I_n - [\lambda \ B_\Gamma^T] \left( \begin{bmatrix}
\lambda^T \\
B_{\Gamma}
\end{bmatrix} 
[\lambda^T \, B_\Gamma^T] \right)^{-1} \begin{bmatrix}
\lambda^T \\
B_{\Gamma}
\end{bmatrix}, \, \Gamma\in \cK(w)
\right\},
\]
where $\cK(w):= \{ \Gamma \subseteq \{1,\ldots,n\} \mid {\rm Supp}(y(w)) \subseteq \Gamma \subseteq \cI\big(\Pi_{C_{\lambda}^\tau}(w)\big) \}$ and 
$y(w)$ is the unique solution to the dual of \eqref{prob:newp}.
We remark here that the nonsingularity of the matrix 
$\begin{bmatrix}
	\lambda^T \\
	B_{\Gamma}
\end{bmatrix} 
[\lambda^T \, B_\Gamma^T]$ follows from Proposition \ref{prop:nondegen} and Remark \ref{remk:licq}; given $w\in\Re^n$, the uniqueness of the dual optimal solution $y(w)$ stems from Corollary \ref{coro:Mystar}. Hence, the set-valued mapping $\cV$ is well-defined. Similarly, the set-valued mapping $\cV$ enjoys similar properties as $\cK$ in Lemma \ref{lem:partialproj}. In particular, it is not difficult to show that $\cV(w) \equiv \partial_B \Pi_{C_{\lambda}^\tau}(w)$ for all $w\in\Re^n$.

Now, we are ready to define the following set-valued mapping associated with $\Pi_{\Delta_{\tau}}$:
\begin{equation}
\label{eq:cS}
\cS(b): = \left\{
S\in\Re^{n\times n} \mid S = P^T V P, \, P\in \Pi^s(b), \, V\in \cV(Pb)
\right\}.
\end{equation}
We show in the next proposition that $\cS$ can be regarded as a computation tractable alternative of the Clarke generalized Jacobian of $\Pi_{\Delta_{\tau}}$ and $\Pi_{\Delta_{\tau}}$ is $\gamma$-order semismooth with respect to $\cS$ for any given $\gamma >0$. 
\begin{proposition}
		\label{prop:cS}
		Let $\lambda\neq 0 \in C$. Then, $\cS$, defined in \eqref{eq:cS}, is a nonempty and  compact valued and upper-semicontinuous set-valued mapping, and for any given vector $b\in\Re^n$, every $S\in \cS(b)$ is symmetric positive semidefinite. Moreover, there exists a neighborhood $W$ of $b$ such that for all $b'\in W$
		\begin{equation*} \label{eq:semismoothPidelta}
		{	\cS(b')\subseteq \cS(b)\quad {\rm and} \quad \Pi_{\Delta_\tau}(b')- \Pi_{\Delta_\tau}(b) - S (b' - b) = 0, \quad \forall\, S \in \cM(b').}
		\end{equation*}
		Therefore, $\partial_B\Pi_{\Delta_{\tau}}(b) \subseteq \cS(b)$.
\end{proposition}
\begin{proof}
	By using the same arguments in \cite[Theorem 1]{Luo2019Solving} and noting the properties of $\cV$, we can easily obtain the desired results.
\end{proof}

Given any $b\in \Re^n$, one can observe that to obtain an element in $\cS(b)$, an important step is to obtain a matrix $V\in \cV(Pb)$ with some $P\in \Pi^s(b)$. We show in the following lemma how this matrix can be efficiently constructed.
\begin{lemma}
	\label{lem:Vcomputation}
	Let $\Gamma \subseteq \{ 1,\ldots, n\}$ be a given index set and  $\alpha\in\Re^n$ be a given nonzero vector. Suppose that $\alpha \not\in {\rm Range}(B)$. It then holds that
	\begin{equation}\label{eq:Scomputation}
	I_n - [\alpha \, B_\Gamma^T] \left( \begin{bmatrix}
	\alpha^T \\
	B_{\Gamma}
	\end{bmatrix} 
	[\alpha \, B_\Gamma^T] \right)^{-1} \begin{bmatrix}
	\alpha^T \\
	B_{\Gamma}
	\end{bmatrix}
	=  H - \frac{1}{\alpha_1^T \alpha_1}\alpha_1 \alpha_1^T, 
	\end{equation}
	where $H = I_n - B_{\Gamma}^T \big( B_\Gamma B_\Gamma^T\big)^{-1} B_\Gamma$ and $\alpha_1\in \Re^n$ denotes the orthogonal projection of $\alpha$ onto the null space of $B_{\Gamma}$, i.e., $\alpha_1 = H\alpha = \alpha - B_{\Gamma}^T \big( B_\Gamma B_\Gamma^T\big)^{-1} B_\Gamma \alpha$.
\end{lemma} 
\begin{proof}
	Since $\alpha \not \in {\rm Rang}(B)$, $\alpha_1$, the orthogonal projection of $\alpha$ onto ${\rm Null}(B_\Gamma)$, is not a zero vector. Moreover, it is not difficult to see that 
	\[ B_\Gamma \alpha_1 = 0 
	\quad {\rm and} \quad
	{\rm Null} \left( 
	\begin{bmatrix}
	\alpha^T \\
	B_{\Gamma}
	\end{bmatrix} 
	\right) = {\rm Null} \left( 
	\begin{bmatrix}
	\alpha_1^T \\
	B_{\Gamma}
	\end{bmatrix} 
	\right) .
	\]
	Hence, by \cite[Lemma 1]{Li2020efficient}, we know that
	\begin{align*}
    [\alpha \, B_\Gamma^T] \left( \begin{bmatrix}
	\alpha^T \\
	B_{\Gamma}
	\end{bmatrix} 
	[\alpha \, B_\Gamma^T] \right)^{-1} \begin{bmatrix}
	\alpha^T \\
	B_{\Gamma}
	\end{bmatrix}
	= {}&   [\alpha_1 \, B_\Gamma^T] \left( \begin{bmatrix}
	\alpha_1^T \\
	B_{\Gamma}
	\end{bmatrix} 
	[\alpha_1 \, B_\Gamma^T] \right)^{-1} \begin{bmatrix}
	\alpha_1^T \\
	B_{\Gamma}
	\end{bmatrix}\\
	={}& [\alpha_1 \, B_\Gamma^T] \begin{bmatrix}
	(\alpha_1^T\alpha_1)^{-1} & \\
	& (B_{\Gamma} B_{\Gamma}^T)^{-1}
	\end{bmatrix}  
	\begin{bmatrix}
	\alpha_1^T \\
	B_{\Gamma}
	\end{bmatrix}\\
	={}& \frac{1}{\alpha_1^T \alpha_1}\alpha_1 \alpha_1^T + B_{\Gamma}^T \big( B_\Gamma B_\Gamma^T\big)^{-1} B_\Gamma.
	\end{align*}
Then, \eqref{eq:Scomputation} follows directly.
\end{proof}
Lemma \ref{lem:Vcomputation} asserts that an element in $\cV(Pb)$ can be easily obtained by performing an additional rank one update to the corresponding matrix $H$. We further note from the discussions in Subsection \ref{subsec:hspic} that the matrix $H$ can be decomposed into the summation of a diagonal matrix and a low rank matrix. Hence, the elements in $\cV(Pb)$ can be constructed in low computational costs and so are the elements in $\cS(b)$ as the corresponding matrix $P$ is nothing but a signed permutation matrix.

\section{Numerical experiments}
\label{sec:numerical}
In this section, we shall evaluate our algorithm for solving problem \eqref{prob:proj}, i.e., the projection onto the OWL1 norm ball. We compare our algorithm with the two best-known methods, i.e., the root-finding scheme in \cite{Zeng2014ordered} and the Davis' approach in \cite{damek2015algorithm} which consists of repeated partial sorting and averaging steps.  While the explanations of Davis' approach may require some new definitions and technical jargons, the root-finding scheme can be easily explained. One will see later that this approach is in fact closely related to our method.

We briefly review the root-finding scheme in \cite{Zeng2014ordered} for solving problem \eqref{prob:proj}. Define the optimal solution mapping $\beta$ as
\begin{equation*}\label{prob:pen}
\beta(\mu):= \argmin_x\left\{ \frac{1}{2} \norm{x - b}^2 
+ \mu \kappa_{\lambda}(x) \right\}, \quad \forall\, \mu \ge 0.
\end{equation*}
and the function $\varphi:\Re_+ \to \Re$ as
$ \varphi(\mu) = \kappa_{\lambda}(\beta(\mu)).$
Then, the root-finding algorithm aims at solving the following univariate nonlinear equation
\begin{equation}
\label{eq:eqroot} 
\varphi(\mu) = \tau,\quad \mu\in [0,\kappa_{\lambda}^\circ(b)], 
\end{equation}
where $\kappa_{\lambda}^\circ$ is the dual norm, or equivalently the polar function, of $\kappa_{\lambda}$, defined as:
\[ 
\kappa_{\lambda}^\circ (y) := \sup_{x} 
\left\{ \inprod{x}{y} \mid \kappa_{\lambda}(x)\le 1 \right\}, \quad \forall\, y\in\Re^n.
\]
From \cite[Lemma 5]{Zeng2014ordered}, we know that there exists a unique solution $\mu^*$ to equation \eqref{eq:eqroot}  on the interval $[0,\kappa_{\lambda}^\circ(b)]$ and $\beta(\mu^*) = \Pi_{\Delta_{\tau}}(b)$, i.e., $\beta(\mu^*)$ solves problem \eqref{prob:proj}. In \cite{Zeng2014ordered}, equation \eqref{eq:eqroot} is solved via the Dekker–Brent method \cite{Dekker1969finding,Brent1973algorithms} and the evaluation of $\beta(\mu)$ for any given $\mu$ relies again on the PAVA.   
From the discussions in \cite{Zeng2014ordered} and Proposition \ref{prop:reform}, we know under certain assumption that equations \eqref{eq:eqroot} and \eqref{prob:first} are in fact optimality conditions associated with equivalent optimization problems. In particular, one may view equation \eqref{eq:eqroot} as a reparameterization of equation \eqref{prob:first}. Hence, the comparisons presented in this section between our method and the root-finding approach can also be regarded as the comparisons between using the semismooth newton method and the Dekker–Brent method for solving the univariate equation \eqref{prob:first}.

We implemented our Algorithm {\sc Ssn} and the root-finding approach \cite{Zeng2014ordered} in {\sc Matlab}.  A highly efficient implementation of the PAVA\footnote{\url{https://statweb.stanford.edu/~candes/software/SortedL1/software.html}} provided in \cite{Bogdan2015Slope} is used to compute $\Pi_C(d)$ for given $d\in\Re^n$ in {\sc Ssn} and to compute $\beta(\mu)$ for given $\mu\ge 0$ in the root-finding approach. As recommended in \cite{Zeng2014ordered}, the highly optimized implementation of the Dekker-Brent method -- the {\sc Matlab} built-in function {\tt fzero} with default parameter values is used for computing the solution to \eqref{eq:eqroot}. Meanwhile, for the Davis' approach, we use the author's {\sc Matlab} mex rapper of his C++ implementation\footnote{\url{https://people.orie.cornell.edu/dsd95/OWLBall.html}}. We also note in order to compute $\Pi_{\Delta_{\tau}}(b)$ for given $\lambda \in C$ and $b\in\Re^n$, all three tested algorithms employ the preprocessing step on $b$ to obtain the sorted and nonnegative input $|b|^{\downarrow}$ and also the postprocessing step to recover $\Pi_{\Delta_{\tau}}(b)$. See Proposition \ref{prop:reform} in the current paper, \cite[Fig. 5.]{Zeng2014ordered} and \cite[Algorithm 1]{damek2015algorithm}.
All our computational results are obtained by running {\sc Matlab} (version 9.5) on a desktop (4-core, Intel Core i7-9700K at 3.60GHz, 16 G RAM).

For our Algorithm {\sc Ssn}, we measure the accuracy of an obtained approximate solution $\tilde y$ for problem \eqref{eq:d} with given data $\lambda, b\in C$ and $\tau >0$ by using the relative residual of the optimality equation \eqref{eq:dphi} as follows:
\[
\eta = \frac{|\inprod{\Pi_C(\tilde y \lambda + b)}{\lambda} - \tau|}{1 + \tau}.
\] 
For a given tolerance $\epsilon >0$, we will stop Algorithm {\sc Ssn} when $\eta < \epsilon$. For all the tests in this section, we set $\epsilon = 10^{-12}$ to demonstrate the finite termination property of Algorithm {\sc Ssn}. In our numerical experiments, it is observed that all three algorithms output almost the same approximate optimal objective value, i.e., the relative difference of these values between any two algorithms is around $10^{-13}$ to $10^{-16}$. Hence, we believe all the algorithms find the ``optimal'' solution and compare only the wall-clock time of each methods.

We test the algorithms with simulated data. In this case, the vector $b\in \Re^n$ is set to be a random vector whose entries are i.i.d. random Gaussian numbers of mean $0$ and std. dev. $\sigma$ with $\sigma \in \left\{10^{-3}, 1, 10^3 \right\}$. Meanwhile, the vector $\lambda \in C$ is obtained by sorting the absolute value of a randomly generated vector in $\Re^n$ using the {\sc Matlab} command:
{\tt lambda=sort(abs(randn(n,1)),'descend')}. To ensure the blank assumption $\kappa_{\lambda}(b) >\tau$, we set the parameter 
\begin{equation*}
\label{eq:tau}
\tau = \beta \kappa_{\lambda}(b)
\end{equation*}
with $\beta$ choosing from $\{ 10^{-3}, 10^{-2}, 10^{-1},0.5,0.8 \}$.
In our tests, we vary $n$ from $10^6$ to $10^8$, i.e., we test algorithms with large-scale problems. 

In Table \ref{table:rand1}, we present our numerical results. The reported results are averaged over $10^2$ (for $n=10^6$ and $n=10^7$) and $10$ (for $n = 10^8$) random realizations.  In the table, the first two columns give parameter $\beta$ and problem dimension $n$.  The number of averaged iterations ``iter.'' and the relative KKT residual $\eta$ of Algorithm {\sc Ssn}, and computation times (in seconds) are listed in the last fifteen columns. In all cases, the best time is in bold. As one can observe, for all the experiments, Algorithm {\sc Ssn} performs best among all three tested algorithms. We note that for almost all the test instances, Algorithm {\sc Ssn} only needs $3$ to $4$ iterations to obtain highly accurate solutions with the relative optimality residual $\eta$ ranging from $10^{-14}$ to $10^{-16}$. These results clearly demonstrate the nice convergence properties, e.g., the quadratic convergence and the finite termination properties, of Algorithm {\sc Ssn}. Careful comparisons show that our Algorithm {\sc Ssn} can be 2 times faster than the root-finding approach and 70 times faster than Davis' approach. Moreover, these results also emphasize that the computational time of our algorithm scales linearly with respect to $n$.

Overall, one can safely conclude that our algorithm {\sc Ssn} is robust and highly efficient for solving projection problems over the OWL1 norm ball. We also note the good performance of the root-finding approach. A possible guess is that the root-finding approach may also enjoy the fast local convergence and the finite termination properties which will be left as a further research topic.

\begin{landscape}
\begin{small}
\begin{center}
	\begin{longtable}{| c | r |  r | r| r| r|r|}
		\caption{The performance of {\sc Ssn}, root-finding approach, and Davis' approach on the  projection problem \eqref{prob:proj}. In the table, $n$ is the dimension of $b$; $\sigma$ is the std. dev. of entries of $b$; $\beta$ is the parameters associated with $\tau$; ``a'' stands for {\sc Ssn}; ``b'' stands for the root-finding approach; ``c'' stands for Davis' approach.}\label{table:rand1}
		\\
		\hline
		\mc{1}{|c|}{} & \mc{2}{|c|}{ }  &\mc{2}{|c|}{ } &\mc{2}{|c|}{ }\\[-5pt]
		\mc{1}{|c|}{} & \mc{2}{|c|}{$\sigma = 10^{-3}$}  &\mc{2}{|c|}{$\sigma = 1$} &\mc{2}{|c|}{$\sigma = 10^{3}$}\\[2pt] \hline
	\mc{1}{|c|}{} &\mc{1}{|c|}{}&\mc{1}{c|}{time}&\mc{1}{|c|}{}&\mc{1}{c|}{time}&\mc{1}{|c|}{}&\mc{1}{c|}{time}
\\ [2pt]\cline{3-3} \cline{5-5} \cline{7-7}
\mc{1}{|c|}{$\beta$ ; $n$} &\mc{1}{|c|}{$\eta$ ; iter.}&\mc{1}{c|}{a $|$ b $|$ c}&\mc{1}{|c|}{$\eta$ ; iter.}&\mc{1}{c|}{a $|$ b $|$ c}&\mc{1}{|c|}{$\eta$ ; iter.}&\mc{1}{c|}{a $|$ b $|$ c}
\\ [2pt]\hline
		\endhead
1e-03 $;$ $10^6$	 &3.3e-14  $;$  4.3 	 &{\bf 0.1 } $|$ 0.2  $|$ 2.0 	 &3.2e-14  $;$  4.3 	 &{\bf 0.1 } $|$ 0.2  $|$ 2.0 	 &2.6e-14  $;$  4.3 	 &{\bf 0.1 } $|$ 0.2  $|$ 2.0 \\[2pt] 
\hline 
1e-02 $;$ $10^6$	 &6.2e-14  $;$  3.7 	 &{\bf 0.1 } $|$ 0.2  $|$ 2.0 	 &8.7e-14  $;$  3.7 	 &{\bf 0.1 } $|$ 0.2  $|$ 2.0 	 &6.4e-14  $;$  3.8 	 &{\bf 0.1 } $|$ 0.2  $|$ 2.0 \\[2pt] 
\hline 
1e-01 $;$ $10^6$	 &4.8e-16  $;$  3.0 	 &{\bf 0.1 } $|$ 0.2  $|$ 2.0 	 &3.8e-16  $;$  3.0 	 &{\bf 0.1 } $|$ 0.2  $|$ 2.1 	 &4.3e-16  $;$  3.0 	 &{\bf 0.1 } $|$ 0.2  $|$ 2.1 \\[2pt] 
\hline 
5e-01 $;$ $10^6$	 &3.3e-16  $;$  3.0 	 &{\bf 0.1 } $|$ 0.2  $|$ 1.8 	 &3.2e-16  $;$  3.0 	 &{\bf 0.1 } $|$ 0.2  $|$ 1.8 	 &2.9e-16  $;$  3.0 	 &{\bf 0.1 } $|$ 0.2  $|$ 1.8 \\[2pt] 
\hline 
8e-01 $;$ $10^6$	 &3.5e-16  $;$  3.0 	 &{\bf 0.1 } $|$ 0.1  $|$ 1.3 	 &3.4e-16  $;$  3.0 	 &{\bf 0.1 } $|$ 0.1  $|$ 1.3 	 &3.2e-16  $;$  3.0 	 &{\bf 0.1 } $|$ 0.2  $|$ 1.3 \\[2pt] 
\hline 
1e-03 $;$ $10^7$	 &4.8e-14  $;$  4.0 	 &{\bf 1.2 } $|$ 2.5  $|$ 34.1 	 &2.7e-14  $;$  4.0 	 &{\bf 1.2 } $|$ 2.5  $|$ 34.1 	 &3.0e-14  $;$  4.0 	 &{\bf 1.2 } $|$ 2.5  $|$ 34.1 \\[2pt] 
\hline 
1e-02 $;$ $10^7$	 &3.3e-14  $;$  3.0 	 &{\bf 1.0 } $|$ 2.0  $|$ 33.8 	 &4.7e-14  $;$  3.0 	 &{\bf 1.0 } $|$ 2.1  $|$ 33.8 	 &1.8e-14  $;$  3.0 	 &{\bf 1.0 } $|$ 2.1  $|$ 33.8 \\[2pt] 
\hline 
1e-01 $;$ $10^7$	 &8.0e-16  $;$  3.0 	 &{\bf 1.1 } $|$ 1.8  $|$ 33.8 	 &1.0e-15  $;$  3.0 	 &{\bf 1.1 } $|$ 1.9  $|$ 33.8 	 &9.4e-16  $;$  3.0 	 &{\bf 1.1 } $|$ 1.9  $|$ 33.7 \\[2pt] 
\hline 
5e-01 $;$ $10^7$	 &1.1e-15  $;$  3.0 	 &{\bf 1.2 } $|$ 1.7  $|$ 29.0 	 &1.0e-15  $;$  3.0 	 &{\bf 1.2 } $|$ 1.8  $|$ 29.0 	 &9.7e-16  $;$  3.0 	 &{\bf 1.2 } $|$ 1.8  $|$ 29.0 \\[2pt] 
\hline 
8e-01 $;$ $10^7$	 &7.3e-16  $;$  3.0 	 &{\bf 1.1 } $|$ 1.5  $|$ 21.4 	 &7.8e-16  $;$  3.0 	 &{\bf 1.1 } $|$ 1.5  $|$ 21.4 	 &9.3e-16  $;$  3.0 	 &{\bf 1.1 } $|$ 1.6  $|$ 21.4 \\[2pt] 
\hline 
1e-03 $;$ $10^8$	 &8.5e-14  $;$  3.9 	 &{\bf 12.2 } $|$ 24.2  $|$ 717.9 	 &8.4e-14  $;$  3.9 	 &{\bf 12.2 } $|$ 24.0  $|$ 717.9 	 &7.6e-14  $;$  3.9 	 &{\bf 12.3 } $|$ 24.2  $|$ 718.3 \\[2pt] 
\hline 
1e-02 $;$ $10^8$	 &5.4e-15  $;$  3.0 	 &{\bf 10.5 } $|$ 20.2  $|$ 743.8 	 &3.4e-15  $;$  3.0 	 &{\bf 10.5 } $|$ 20.4  $|$ 743.7 	 &3.7e-15  $;$  3.0 	 &{\bf 10.4 } $|$ 20.6  $|$ 756.6 \\[2pt] 
\hline 
1e-01 $;$ $10^8$	 &1.8e-15  $;$  3.0 	 &{\bf 11.3 } $|$ 17.3  $|$ 752.8 	 &3.1e-15  $;$  3.0 	 &{\bf 11.2 } $|$ 17.8  $|$ 764.7 	 &2.8e-15  $;$  3.0 	 &{\bf 11.2 } $|$ 17.8  $|$ 743.6 \\[2pt] 
\hline 
5e-01 $;$ $10^8$	 &3.2e-15  $;$  3.0 	 &{\bf 12.9 } $|$ 18.0  $|$ 700.9 	 &3.9e-15  $;$  3.0 	 &{\bf 12.8 } $|$ 18.7  $|$ 708.0 	 &3.0e-15  $;$  3.0 	 &{\bf 12.8 } $|$ 19.3  $|$ 711.7 \\[2pt] 
\hline 
8e-01 $;$ $10^8$	 &8.8e-14  $;$  2.9 	 &{\bf 12.1 } $|$ 15.9  $|$ 591.0 	 &2.9e-15  $;$  3.0 	 &{\bf 12.9 } $|$ 16.2  $|$ 590.1 	 &3.1e-15  $;$  3.0 	 &{\bf 12.2 } $|$ 17.3  $|$ 574.9 \\[2pt] 
\hline 
	\end{longtable}
\end{center}
\end{small}
\end{landscape}

\section{Conclusion}
In this paper, we introduced an finitely terminated yet efficient semismooth Newton method to compute the projector over the OWL1 norm ball. Numerical experiments demonstrate the efficiency of the algorithm. We also conducted a thorough first order variational analysis on the projector $\Pi_{\Delta_{\tau}}$. In particular, the generalized Jacobian associated with $\Pi_{\Delta_{\tau}}$ is given and its critical sparse plus low rank structure is uncovered. We believe that these results are critical for designing efficient and robust algorithms for solving large-scale general OWL1 norm constraint problems which will be considered in our future work.

\section{Appendix}
To start the proof for Proposition \ref{prop:reform}, we need the following two reduction lemmas. 
\begin{lemma}
	\label{lem:Pb}
For any $b\in \Re^n$, it holds that 
\begin{equation}
\label{prob:Pbproj}
P\Pi_{\Delta_{\tau}}(b) = \Pi_{\Delta_{\tau}}(Pb) = \argmin_x 
\left\{
\frac{1}{2}\norm{x - Pb}^2 \mid \kappa_{\lambda}(x)\le \tau
\right\}
\end{equation}
for all $P\in \Pi^s(b)$.
\end{lemma}
\begin{proof}
Recall the definition of $\Pi^s(b)$ in \eqref{eq:Pmatrix}, we know that $Pb = |b|^\downarrow$ for any $P\in \Pi^s(b)$. Hence, $\kappa_{\lambda}(P\Pi_{\Delta_{\tau}}(b)) = \inprod{|\Pi_{\Delta_{\tau}}(b)|^\downarrow}{\lambda} = \kappa_{\lambda}(\Pi_{\Delta_{\tau}}(b)) \le \tau$, i.e., $P\Pi_{\Delta_{\tau}}(b)$ is a feasible solution to problem \eqref{prob:Pbproj}.

For any $P\in \Pi^s(b)$ and $y\in\Re^n$ satisfying $\kappa_{\lambda}(y)\le \tau$, we have that 
\[ \frac{1}{2}\norm{y-Pb}^2 = \frac{1}{2}
\norm{P(P^Ty - b)}^2 = \frac{1}{2}\norm{P^Ty-b}^2 
\ge \frac{1}{2}\norm{\Pi_{\Delta_{\tau}}(b) - b}^2 
= \frac{1}{2}\norm{P\Pi_{\Delta_{\tau}}(b) - Pb}^2, \]
i.e., $P\Pi_{\Delta_{\tau}}(b)$ is an optimal solution to problem \eqref{prob:Pbproj} for all $P\in \Pi^s(b)$.
\end{proof}
\begin{lemma}
	\label{lem:binC}
Suppose that $\kappa_{\lambda}(b) > \tau$. Then, it holds that $\Pi_{\Delta_{\tau}}(b) = \Pi_{C_{\lambda}^\tau}(b)$ for all $b\in C$.
\end{lemma}
\begin{proof}
Given any $b\in C$, we first show that $(\Pi_{\Delta_{\tau}}(b))_i\ge 0$ for all $1\le i \le n$. Suppose otherwise that there exists $i_0\in \{1,\ldots,n\}$ such that $(\Pi_{\Delta_{\tau}}(b))_{i_0} < 0$. Let $u\in\Re^n$ be the vector such that $u_{i_0} = -(\Pi_{\Delta_{\tau}}(b))_{i_0}$ and $u_i = (\Pi_{\Delta_{\tau}}(b))_i$ for all $i\neq i_0$. Then, it holds that
$u\neq \Pi_{\Delta_{\tau}}(b)$ and $\kappa_\lambda(\Pi_{\Delta_{\tau}}(b)) = \kappa_{\lambda}(u)$. Meanwhile, we have that 
\[
\frac{1}{2}\norm{\Pi_{\Delta_{\tau}}(b) - b}^2 - \frac{1}{2}\norm{u - b}^2 = 2u_{i_0} b_{i_0} \ge 0,
\]
which contradicts to the fact that $\Pi_{\Delta_{\tau}}(b)$ is the unique optimizer of problem \eqref{prob:proj}. Thus, $(\Pi_{\Delta_{\tau}}(b))_i\ge 0$ for all $1\le i \le n$.
	
Next, we show that for any $1\le i < j \le n$, it holds that $(\Pi_{\Delta_{\tau}}(b))_i \ge (\Pi_{\Delta_{\tau}}(b))_j$. Suppose otherwise that there exists $1\le i_0 < j_0\le n$ such that  $(\Pi_{\Delta_{\tau}}(b))_{i_0} < (\Pi_{\Delta_{\tau}}(b))_{j_0}$. Then, we can construct a new vector $u\in\Re^n$ by switching the $i_0$ and $j_0$ entries of $\Pi_{\Delta_{\tau}}(b)$, i.e., $u_{i_0} = (\Pi_{\Delta_{\tau}}(b))_{j_0}$, $u_{j_0} = (\Pi_{\Delta_{\tau}}(b))_{i_0}$ and $u_i = (\Pi_{\Delta_{\tau}}(b))_{i}$ for all $i \in \left\{1,\ldots,n \right\}/\{i_0,j_0\}$. Then, $u\neq \Pi_{\Delta_{\tau}}(b)$ and 
$\kappa_{\lambda}(u) = \kappa_{\lambda}(\Pi_{\Delta_{\tau}}(b))\le \tau$. Similarly, one can show by some simple calculations that
\[
\frac{1}{2}\norm{\Pi_{\Delta_{\tau}}(b) - b}^2 - \frac{1}{2}\norm{u - b}^2 \ge 0.
\]
We again arrive at a contradiction. Hence, we have $\Pi_{\Delta_{\tau}}(b)\in C$.

Since $\Pi_{\Delta_{\tau}}(b)\in C$ and $\kappa(b) > \tau$, we know that $\inprod{\lambda}{\Pi_{\Delta_{\tau}}(b)} = \kappa_{\lambda}(\Pi_{\Delta_{\tau}}(b)) = \tau$, i.e., $\Pi_{\Delta_{\tau}}(b)$ is a feasible solution to problem \eqref{prob:newp}. Hence, 
\[
\frac{1}{2}\norm{\Pi_{C_{\lambda}^\tau}(b) - b}^2 \le \frac{1}{2}\norm{\Pi_{\Delta_{\tau}}(b) - b}^2 \le \frac{1}{2}\norm{\Pi_{C_{\lambda}^\tau}(b) - b}^2,
\]
where the last inequality follows from the fact that $\Pi_{C_{\lambda}^\tau}(b)$ is a feasible solution to problem \eqref{prob:proj}. Hence, the uniqueness of the optimal solution of problem \eqref{prob:newp} implies that $\Pi_{\Delta_{\tau}}(b) = \Pi_{C_{\lambda}^\tau}(b)$.
\end{proof}

\noindent{\bf Proof of Proposition \ref{prop:reform}}: 
For any $b\in \Re^n$ satisfying $\kappa_{\lambda}(b) > \tau$ and $P\in \Pi^s(b)$, we know from Lemmas \ref{lem:Pb} and \ref{lem:binC} that 
\[P\Pi_{\Delta_{\tau}}(b) = \Pi_{\Delta_{\tau}}(Pb) = \Pi_{C_{\lambda}^\tau}(Pb).
\] 
Since $P$ is a signed permutation matrix, we further have
$\Pi_{\Delta_{\tau}}(b) = P^T  \Pi_{C_{\lambda}^\tau}(Pb)$. \qed

\bibliographystyle{siam}

\end{document}